\documentclass[10pt,leqno]{article}
\usepackage[affil-it]{authblk}
\usepackage{amssymb}
\usepackage{amsmath}
\usepackage{amsrefs}
\usepackage{tabularx}
\usepackage[arrow,matrix,tips,frame,color,line,poly]{xy}
\usepackage{theorem}
\newtheorem{theorem}[equation]{Theorem}

\newtheorem{definition}[equation]{Definition}

\newtheorem{lemma}[equation]{Lemma}
\newtheorem{conjecture}[equation]{Conjecture}
\newtheorem{proposition}[equation]{Proposition}
{\theorembodyfont{\rmfamily}

\newtheorem{remark}[equation]{Remark}

}
\newcommand{\qed}{\hfill $\square$ \medskip}

\newenvironment{proof}[1][Proof]{\noindent\textbf{#1.} }{\
\qed}
\renewcommand{\exp}{\mathrm{exp}}
\newcommand{\wt}{\widetilde}

\newcommand{\caP}{\mathcal{P}}
\newcommand{\caT}{\mathcal{T}}

\newcommand{\T}{\mathcal T}
\newcommand{\X}{\mathcal X}

\newcommand{\R}{\mathbb R}

\newcommand{\C}{\mathbb C}
\newcommand{\Z}{\mathbb Z}

\newcommand{\E}{\mathbb E}
\newcommand{\G}{\mathbb G}

\newcommand{\ch}[1]{\negthinspace\negthinspace\negthinspace\phantom{a}^\vee\negthinspace #1}

\newcommand{\h}{\mathfrak h}

\newcommand\inv{^{-1}}

\newcommand{\g}{\mathfrak g}

\newcommand{\socle}{\mathrm{socle}}
\newcommand{\cosocle}{\mathrm{cosocle}}

\DeclareMathOperator\iNT{int}
\DeclareMathOperator\Int{Int}
\DeclareMathOperator\Hom{Hom}
\DeclareMathOperator\diag{diag}

\DeclareMathOperator\Out{Out}
\DeclareMathOperator\Aut{Aut}
\DeclareMathOperator\Ind{Ind}
\DeclareMathOperator\Cent{Cent}
\DeclareMathOperator\Norm{Norm}
\DeclareMathOperator\sgn{sgn}

\DeclareMathOperator\RE{Re}

\renewcommand{\sec}[1]{\section{#1}
\renewcommand{\theequation}{\thesection.\arabic{equation}}
  \setcounter{equation}{0}}
\newcommand{\subsec}[1]{\subsection{#1}
\renewcommand{\theequation}{\thesection.\arabic{equation}}}

\DeclareMathOperator\Lie{Lie}
\newcommand{\Grad}{G_{\text{rad}}}

\newcommand{\chirad}{\chi_{\text{rad}}}
\newcommand{\chiinf}{\chi_{\text{inf}}}
\renewcommand{\L}[1]{\overset{L}{\vphantom{a}}\negthinspace #1}
\renewcommand{\E}[1]{\overset{E}{\vphantom{a}}\negthinspace #1}
\renewcommand{\d}[1]{\overset{d}{\vphantom{a}}\negthinspace #1}
\newcommand{\LG}{\overset{L}{\vphantom{a}}\negthinspace G}

\newcommand{\LH}{\L H}
\newcommand{\EH}{\E H}

\begin{document}

\title{Contragredient representations and characterizing the local
     Langlands correspondence}
\author{Jeffrey Adams\\jda@math.umd.edu}
\affil{Department of Mathematics\\
University of Maryland\\
College Park, MD 20742\\
}
\author{David A. Vogan Jr.\\
dav@math.mit.edu}
\affil{Department of Mathematics\\
Massachusetts Institute of Technology\\
Cambridge, MA 02139}
\maketitle

{\renewcommand{\thefootnote}{} 
\footnote{Supported in part by  National Science
Foundation Grants \#DMS-0967566 and \#DMS-1317523 (first author) and
\#DMS-0968275 (both authors)
}}

\sec{Introduction}
It is surprising that the following question has not been addressed in
the literature:
what is the contragredient in terms of
Langlands parameters?

Thus suppose $G$ is 
a connected, reductive algebraic group
defined over a local field $F$, and $G(F)$ is its $F$-points.
According to the local Langlands conjecture, associated to an admissible
homomorphism $\phi$ from the Weil-Deligne group of $F$ into the
L-group of $G(F)$ is an L-packet $\Pi(\phi)$, a finite set of (equivalence classes of)
irreducible admissible representations of $G(F)$. (If $F$ is archimedean, ``equivalence'' means ``infinitesimal equivalence.'' If $F$ is non-archimedean, it means ``equivalence of smooth vectors.'') Conjecturally these
L-packets partition the admissible dual.

So suppose $\pi$ is an irreducible admissible representation, and
$\pi\in \Pi(\phi)$. Let $\pi^*$ be the contragredient, or dual, of
$\pi$ (see \eqref{e:contragredient}). The question is: what is the
homomorphism $\phi^*$ such that $\pi^*\in\Pi(\phi^*)$?  We also
consider the related question of describing the {\it Hermitian dual}
in terms of Langlands parameters.  Both of the questions come down to
a characterization of the local Langlands correspondence. For $\R$
this is the content of Sections \ref{s:lpackets} and \ref{s:lparameters}, especially
Definitions \ref{d:dslpacket} and \ref{d:general}.

Let $\ch G$ be the complex dual group of $G$.
A Chevalley involution $C$ of $\ch G$  
satisfies $C(h)=h\inv$, for all $h$ in some Cartan subgroup 
of $\ch G$.
The L-group $\LG$ of $G(F)$ is a certain semidirect product 
$\ch G\rtimes\Gamma$ where $\Gamma$ is the absolute Galois group of $F$
(or other related groups).
We can choose $C$ so that
it extends to an involution of $\L G$,
acting trivially on 
$\Gamma$. We refer to this as the Chevalley involution of $\LG$. See Section \ref{s:aut}.

We believe the  contragredient should correspond to composition with the Chevalley involution
of $\LG$. To avoid two levels of conjecture, we formulate this as
follows. 

\begin{conjecture}
\label{c:main}
Assume the local Langlands conjecture is known for both $\pi$ and
$\pi^*$. Let $C$ be the Chevalley involution of $\LG$. 
Then
$$
\pi\in\Pi(\phi)\Leftrightarrow \pi^*\in\Pi(C\circ\phi).
$$
\end{conjecture}

Even the following weaker result is not known:
\begin{conjecture}
\label{c:second}
If $\Pi$ is an L-packet, then so is $\Pi^*=\{\pi^*\,|\, \pi\in\Pi\}$.
\end{conjecture}

\medskip

The local Langlands conjecture is only known, for fixed $G(F)$ and all
$\pi$, in a limited number cases, notably $GL(n,F)$ over any local
field \cite{harris_taylor}, \cite{henniart_llc}, and for any $G$ if
$F=\R$ or $\C$. See Langlands's original paper
\cite{langlandsClassification}, which is summarized in Borel's article
\cite{borelCorvallis}.  On the other hand it is known for a restricted
class of representations for more groups, for example unramified
principal series representations of a split p-adic group
\cite{borelCorvallis}*{10.4}.  

For discrete series L-packets of real groups, Conjecture \ref{c:main} 
follows from \cite{shelstad_structure}*{Lemma 7.4.1}, 
and it is known 
for some representations of some
quasisplit $p$-adic groups by 
\cite{kaletha_genericity}.
It would be reasonable to impose Conjecture \ref{c:main} as a
condition on the local Langlands correspondence in cases where it is
not known.

We concentrate on the Archimedean case.
Let $W_\R$ be the Weil group of $\R$. The contragredient in this case
can be realized either via the Chevalley automorphism of $\LG$, or via
a similar automorphism of $W_\R$. Let $\C^*$ be the multiplicative group 
of nonzero complex numbers. 
By analogy with the Chevalley involution $C$ of $G$, there is a unique $\C^*$-conjugacy class of
automorphisms $C_{W_\R}$ of $W_\R$ satisfying $C_{W_\R}(z)=z\inv$ for all
$z\in\C^*$. 
See Section \ref{s:aut}. 

\begin{theorem}
\label{t:main}
Let $G(\R)$ be the real points of a connected reductive algebraic
group defined over $\R$, with L-group $\LG$.
Suppose $\phi\colon W_\R\rightarrow \LG$ is an admissible homomorphism, 
with associated L-packet $\Pi(\phi)$. Then
$$
\Pi(\phi)^*=\Pi(C\circ\phi)=\Pi(\phi\circ C_{W_\R}).
$$
In particular $\Pi(\phi)^*$ is an L-packet.
\end{theorem}

Here is a sketch of the proof.

It is easy to prove in the case of tori. See Section \ref{s:tori}. 

It is well known that an L-packet $\Pi$ of (relative) discrete series
representations is determined by an infinitesimal and a central
character. In fact something stronger is true.
Let $\Grad$ be the radical of $G$, i.e., the maximal central torus.
Then $\Pi$
is determined by an
infinitesimal character and a
character of $\Grad(\R)$, which we refer to as a radical character.

In general it is easy to read off the infinitesimal
and radical characters  of $\Pi(\phi)$, see  \cite{borelCorvallis} 
and Section \ref{s:chars}.
In particular for a relative discrete series parameter Theorem  \ref{t:main}
reduces to a claim about how $C$ affects the infinitesimal 
and radical characters.
For the radical 
character this reduces to the case of tori, and Theorem \ref{t:main}
follows  in this case.

This is the heart of the matter, and the general case follows easily
by parabolic induction. In other words, the proof relies on the fact
that the parameterization is  uniquely characterized by:
\begin{enumerate}
\item Infinitesimal character,
\item Radical character,
\item Compatibility with parabolic induction.
\end{enumerate}
In a sense this is the main result of the paper: a self-contained
description of the local Langlands classification, and its characterization by (1-3).
See Sections \ref{s:lpackets} and \ref{s:lparameters}.
Use of the Tits group (see Section \ref{s:tits}) 
simplifies some technical arguments.

Now consider  $GL(n,F)$ for $F$ a local field of characteristic $0$.
Since $GL(n,F)$ is split, an admissible homomorphism into the L-group
is the same thing as a homomorphism into the dual group $GL(n,\C)$ (in
which the Weil group acts semisimply); that is, 
$\phi$ may be identified with an $n$-dimensional complex representation of the
Weil-Deligne group $W'_F$. 
In this case L-packets are singletons, so write $\pi(\phi)$ for the
representation attached to $\phi$.

For the Chevalley involution take $C(g)=\phantom{}^tg\inv$. Composing
any finite-dimensional representation into $GL(n)$ with the inverse
transpose gives the contragredient representation; so
$C\circ\phi$ is equivalent to the contragredient $\phi^*$ of $\phi$.
Over $\R$ Theorem \ref{t:main} says that the
Langlands correspondence commutes with the contragredient: 
\begin{equation}
\label{e:dualgln}
\pi(\phi^*)\simeq \pi(\phi)^*.
\end{equation}
This is also true over a $p$-adic field \cite{harris_taylor},
\cite{henniart_llc}, in which case it  is closely related to the
functional equations for L and $\varepsilon$ factors.

\medskip

We now consider a variant of \eqref{e:dualgln} in the real
case. Suppose $\pi$ is an irreducible representation of $GL(n,\R)$.
Its Hermitian dual $\pi^h$ is the unique irreducible representation,
admitting a nonzero invariant sesquilinear pairing with $\pi$. The
representation $\pi$ admits a nonzero invariant Hermitian form if and
only if $\pi\simeq \pi^h$. In this case we say $\pi$ is Hermitian.
See Section \ref{s:hermitian}.

The Hermitian dual arises naturally in the study of unitary
representations: the unitary dual is the subset of the fixed points of
the involution $\pi \mapsto \pi^h$, consisting of those $\pi$ for
which the invariant form is definite. So it is natural to ask what the
Hermitian dual is on the level of Langlands parameters.

There is a natural notion of Hermitian dual of a finite-dimensional
representation $\phi$ of any group:
$\phi^h=\phantom{}^t\overline{\phi}\,\inv$, and 
$\phi$ preserves a nondegenerate Hermitian form
if and only if $\phi\simeq\phi^h$.

The local Langlands correspondence for $GL(n,\R)$ commutes with the
Hermitian dual operation:

\begin{theorem}
\label{t:hermitian}
Suppose $\phi$ is an n-dimensional semisimple  representation of $W_\R$. 
Then:
\begin{enumerate}
\item $\pi(\phi^h)=\pi(\phi)^h$,
\item $\phi$ is Hermitian  if and only if $\pi(\phi)$ is Hermitian,
\item $\phi$ is unitary if and only if $\pi(\phi)$ is tempered.
\end{enumerate}

\end{theorem}
See Section \ref{s:hermitian}. 
\medskip

Return now to the setting of general real groups.
The space  $\X_0$ of conjugacy classes of L-homomorphisms
parametrizes L-packets of representations. By introducing some extra
data we obtain a space $\X$ which parametrizes irreducible
representations \cite{abv}.  Roughly speaking $\X$ is the set of
conjugacy classes of 
pairs $(\phi,\chi)$ where $\phi\in\X_0$ and $\chi$ is a character of
the component group of $\Cent_{\ch G}(\phi(W_\R))$. It is natural to ask
for the involution of $\X$ induced by the contragredient.

On the other hand, it is possible to formulate and prove an analogue
of Theorem \ref{t:hermitian} for general real groups,
in terms of an antiholomorphic involution of $\LG$. Also, the
analogue of Theorem \ref{t:hermitian}  holds in the $p$-adic case.
All of these topics require  more machinery.
In an effort to keep the presentation as elementary as possible we defer 
them to later papers.

This paper is a complement to \cite{real_chevalley}, which 
considers the action of the Chevalley involution of $G$,
rather than $\LG$. See Remark \ref{r:dual}.

We thank Kevin Buzzard for asking about the contragredient on the
level of L-parameters. We also thank the referee for carefully reading
the paper and making number of suggestions, such as changing the
title, which substantially improved the exposition.

\sec{The Chevalley Involution}
\label{s:aut}

We discuss the Chevalley involution. 
This is well known, although there isn't a good reference for
all of the details we need (Chevalley cites the existence of this
automorphism without proof in \cite{chevalley_simples}*{page 23}).
For the convenience of the reader we give
complete details. We also discuss a similar involution of $W_\R$.

Throughout this paper $G$ is a connected, reductive algebraic group. 
We may identify it with its complex points, and write $G(\C)$ on occasion
to emphasize this point of view. 
For $x\in G$ write $\iNT(x)$ for the inner automorphism
$\iNT(x)(g)=xgx\inv$. 

\begin{proposition}
\label{p:CH}
Fix a Cartan subgroup $H$ of $G$.
There is an automorphism $C$ of $G$ 
satisfying $C(h)=h\inv$ for all $h\in H$.
For any such automorphism $C^2=1$, and, 
for every  semisimple element $g\in G$, $C(g)$ is conjugate to
$g\inv$.

Suppose $C_1$ and $C_2$ are two such automorphisms defined with respect to
Cartan subgroups $H_1$ and $H_2$. Then $C_1$ and $C_2$ are conjugate
by an inner automorphism of $G$. 
\end{proposition}

The proof uses {\em based root data} and {\em pinnings}. 
For background see \cite{springer_book}.
Fix a Borel subgroup $B$ of $G$, and a Cartan subgroup $H\subset
B$. Let $X^*(H), X_*(H)$ be the character and co-character lattices of
$H$, respectively.  Let $\Pi,\ch\Pi$ be the sets of simple
roots, respectively simple co-roots, defined by $B$. 

The {\em based root datum   defined by $(B,H)$} is
$(X^*(H),\Pi,X_*(H),\ch\Pi)$.
There is a natural notion of isomorphism of based root data. 
A  {\em pinning}  is a set
$\caP=(B,H,\{X_\alpha|\alpha\in\Pi\})$ where, 
for each $\alpha\in \Pi$, $X_\alpha\ne 0$ is contained in 
the $\alpha$-root space $\g_\alpha$ of 
$\g=\Lie(G)$. Let $\Aut(G)$ be the group of algebraic (equivalently,
holomorphic) automorphisms of  
$G$, $\Int(G)\subset \Aut(G)$ the inner automorphisms, and
$\Out(G)=\Aut(G)/\Int(G)$.  
Let $\Aut(\caP)$ be the subgroup of $\Aut(G)$ preserving
$\caP$.
We refer to the elements of $\Aut(\caP)$ as {\em $\caP$-distinguished
automorphisms}.

\begin{theorem}[\cite{springer_book}*{Theorem 9.6.2}]
\label{t:auts}
Suppose $G,G'$ are connected, reductive complex groups. Fix
pinnings $\caP=(B,H,\{X_\alpha\})$ and $\caP'=(B',H',\{X'_\alpha\})$.
Let $D_b,D_b'$ be the based root data defined by $(B,H)$ and
$(B',H')$. 

Suppose $\phi\colon D_b\rightarrow D'_b$ is an isomorphism of
based root data. Then there is a unique isomorphism
$\psi\colon G\rightarrow G'$
taking $\caP$ to $\caP'$ and inducing $\phi$ on the root
data. 

The only inner automorphism in $\Aut(\caP)$ is the
identity, and there  are isomorphisms
\begin{equation}
\label{e:outaut}
 \Out(G)\simeq \Aut(D_b)\simeq \Aut(\caP)\subset \Aut(G).
\end{equation}
\end{theorem}
The following consequence of the theorem is quite useful.

\begin{lemma}
\label{l:inth}
Suppose $\tau\in\Aut(G)$ restricts trivially to a Cartan
subgroup $H$. Then $\tau=\iNT(h)$ for some $h\in H$.
\end{lemma}

\begin{proof}
Fix a pinning
$(B,H,\{X_\alpha\})$. 
Then $d\tau(\g_\alpha)=\g_\alpha$ for all  $\alpha$. 
Therefore we can choose 
$h\in H$ so that $d\tau(X_\alpha)=\text{Ad}(h)(X_\alpha)$ for
all $\alpha\in\Pi$. Then $\tau\circ\iNT(h\inv)$ acts trivially on
$D_b$ and $\caP$. By the theorem $\tau=\iNT(h)$.
\end{proof}

\begin{proof}[Proof of the Proposition]
Choose a Borel subgroup $B$ containing $H$
and let $D_b=(X^*(H),\Pi,X_*(H),\ch\Pi)$ 
be the based root datum defined by $(B,H)$.
Let $B^{op}$ be the opposite Borel, 
with corresponding root datum
$D_b^{op}=(X^*(H),-\Pi,X_*(H),-\ch\Pi)$.

Choose a pinning
$\caP=(B,H,\{X_\alpha\})$.
Let $\mathcal P^{op}=(H,B^{op},\{X_{-\alpha}|\alpha\in\Pi\})$ where, 
for each simple root $\alpha\in\Pi$,  the root vector
$X_{-\alpha}\in\g_{-\alpha}$ is determined by the requirement 
$\alpha([X_\alpha,X_{-\alpha}])=2$.

Let $\phi:D_b\rightarrow D'_b$ be the isomorphism of based
root data given by $-1$ on $X^*(H)$. By Theorem \ref{t:auts} there
is an automorphism $C_{\caP}$ of $G$ taking $\caP$ to $\caP^{op}$ and
inducing $\phi$. In particular $C_{\caP}(h)=h\inv$ for $h\in H$.
This implies $C_{\caP}(\g_\alpha)=\g_{-\alpha}$, and since
$C_{\caP}\colon \caP\rightarrow \caP^{op}$ we have $C_{\caP}(X_\alpha)=X_{-\alpha}$. 
Since $C_{\caP}^2$ is an automorphism of $G$ taking $\caP$  to
itself and inducing the
trivial automorphism of $D_b$, the theorem implies that xs$C_{\caP}^2=1$. 

If $g\in G$ is any semisimple element, choose $x$ so that $xgx\inv\in H$. 
Then $C(g)= (C(x\inv)x)g\inv(C(x\inv) x)\inv$.

Suppose $C_1(h)=h\inv$ for all $h\in H$. Then 
$C_1\circ C_\caP$ acts trivially on $H$, so by the lemma 
$C_1=\iNT(h_1)\circ C_\caP$ for some $h_1\in H$, which implies $C_1^2=1$.

For the final assertion choose $g\in G$ so that $gH_1g\inv=H_2$. 
Then $\iNT(g)\circ C_1\circ \iNT(g\inv)$ acts by inversion on $H_2$. 
By the lemma 
$\iNT(g)\circ C_1\circ \iNT(g\inv)=\iNT(h_2)\circ C_2$ for some $h_2\in
H_2$. Choose $t\in H_2$ so that $t^2=h_2$.  Then $\iNT(t\inv g)\circ C_1\circ
\iNT(t\inv g)\inv=C_2$. 
\end{proof}

An involution satisfying the condition of the proposition  is known as a
{\em Chevalley involution}.
For $\caP$ a pinning we refer to the involution $C_\caP$ of the
proof as the {\em Chevalley involution defined by $\caP$}.
The proof shows that every Chevalley involution is equal to
$C_\caP$ for some $\caP$, and all Chevalley involutions are
conjugate. We will abuse terminology slightly and refer to {\em the} Chevalley
involution. 

\begin{remark}
\hfil
\begin{enumerate}
\item The Chevalley involution satisfies: $C(g)$ is conjugate to
$g\inv$ for {\it all} $g\in G$ \cite{lusztig_remarks}*{Proposition 2.6}.
(Lusztig proves the corresponding statement over algebraically closed
fields of arbitrary characteristic.) 

\item If $G=GL(n)$,
$C(g)=\negthinspace\negthinspace\negthinspace\phantom{a}^tg\inv$ is a
Chevalley involution.
The group of fixed points is $G^C=O(n,\C)$, the complexified maximal
compact subgroup 
of $GL(n,\R)$. In other words, $C$ is the Cartan involution for
$GL(n,\R)$.  In general the Chevalley involution is the
Cartan involution of the split real form of $G$.

\item Suppose $C'$ is any automorphism such that
$$\text{$C'(g)$ is $G$-conjugate to $g\inv$ for all semisimple $g$ in 
  $G$.} \eqno{(*)}$$
It is not hard to see, using Lemma \ref{l:inth}, that 
$C'=\iNT(x)\circ C$ for some $x\in G$ and
some Chevalley involution $C$. This means that the requirement $(*)$ characterizes (via the canonical map 
  $\Aut(G)\rightarrow\Out(G)$) the class of the Chevalley involutions
  in $\Out(G)$. 
\item The Chevalley involution is inner if and only if $G$ is semisimple and
$-1$ is in the Weyl group, in which case 
$C=\iNT(g_0)$ where $g_0\in\Norm_G(H)$ represents $-1$.
The proposition implies $g_0^2$ is central, and 
independent of all choices. See Lemma \ref{l:titslemma}.
\end{enumerate}

\end{remark}

\begin{lemma}
\label{l:commutes}
Fix a pinning $\caP$. Then $C_\caP$ commutes with every
$\caP$-dis\-tin\-guished automorphism.
\end{lemma}
This is immediate from the uniqueness statement in Theorem \ref{t:auts}.

Here is a similar involution of $W_\R$.
Recall $W_\R=\langle \C^*,j\rangle$ with relations $jzj\inv=\overline
z$ and $j^2=-1$. 

\begin{lemma}
\label{l:tau}
There is an involution $C_{W_\R}$ of $W_\R$ such that $C_{W_\R}(z)=z\inv$ for
all $z\in\C^*$.
Any two such automorphisms
are conjugate by $\iNT(z)$ for some 
$z\in\C^*$.
\end{lemma}

\begin{proof}
This is elementary. 
For $z_0\in\C^*$ define  $C_{W,z_0}(z)=z\inv$ $(z\in\C^*$) and $C_{W,z_0}(j)=z_0j$.
From the relations  this extends to  an automorphism of $W_\R$ if and only if
$z_0\overline z_0=1$. Thus $C_{W,1}$ is an automorphism,
and 
$C_{W,z_0}=\iNT(u)\circ C_{W,1}\circ\iNT(u\inv)$,
provided
$(u/\overline u)^2=z_0$. 
\end{proof}

\sec{Tori}
\label{s:tori}

\begin{subequations}\label{se:cover}
Let $H$ be a complex torus, and fix an element 
$\gamma\in \frac12X^*(H)$. 
Let 
\begin{equation}
\label{e:Hgamma}
H_\gamma=\{(h,z)\in H\times\C^*\,|\, [2\gamma](h)=z^2\}.
\end{equation}
This is a two-fold cover of $H$ via the map
$(h,z)\rightarrow h$; write $\zeta$ for the nontrivial element in the
kernel of this map. 
We call this the {\em $\gamma$-cover} of $H$. A character of
$H_\gamma$ is called {\em genuine} if it takes the value $-1$ on $\zeta$.
Note that  
\begin{equation}
\gamma\colon H_\gamma \rightarrow \C^*, \quad (h,z)\mapsto z
\end{equation}
is a genuine character of $H_\gamma$, and is a canonical square root of the
algebraic character $2\gamma$ of $H$. The genuine algebraic characters
of $H_\gamma$ may be identified with symbols $\gamma+X^*(H)$:
\begin{equation}
\gamma + \kappa_0 \colon H_\gamma \rightarrow \C^*, \quad (h,z)\mapsto
\kappa_0(h)z \qquad (\kappa_0 \in X^*(H)).
\end{equation}

Now assume $H$ is defined over $\R$, with Cartan involution $\theta$. 
The {\em $\gamma$-cover of $H(\R)$} is defined to be the inverse image
of $H(\R)$ in 
$H_\gamma$. 

\end{subequations} 
\begin{lemma}[\cite{av1}*{Proposition 5.8}]
\label{l:lambdakappa}
Given $\gamma\in \frac12 X^*(H)$, the genuine characters of $H(\R)_\gamma$ are canonically parametrized
by the set of pairs $(\lambda,\kappa)$ with 
$$\lambda\in \Hom(\h,\C), \qquad
\kappa\in\gamma+X^*(H)/(1-\theta)X^*(H),$$
subject to the requirement $(1+\theta)\lambda=(1+\theta)\kappa$.
\end{lemma}

Write $\chi(\lambda,\kappa)$ for the character defined by
$(\lambda,\kappa)$. This character has differential $\lambda$, and its
restriction to the maximal compact subgroup is the restriction of the
character $\kappa$ of $H_\gamma$. A little more precisely, fix
$\kappa_0\in X^*(H)$ so that $\gamma+\kappa_0$ is a representative of
the $\kappa$. Then the restriction of $\chi(\lambda,\kappa)$ to the
maximal compact subgroup of $H(\R)_\gamma$ is the restriction of the
algebraic character $\gamma+\kappa_0$ defined in \eqref{se:cover}.

Let $\ch H$ be the dual torus. This satisfies: $X^*(\ch H)=X_*(H)$, $X_*(\ch
H)=X^*(H)$. If $H$ is defined over $\R$, with Cartan involution
$\theta$, then $\theta$ may be viewed as an involution of $X_*(H)$;  its adjoint
$\theta^t$ is an involution of $X^*(H)=X_*(\ch H)$. Let
$\ch\theta$ be the automorphism of $\ch H$ induced by $-\theta^t$.

The {\em L-group of $H$} is defined as $\LH=\langle \ch H,\ch\delta\rangle$
where $\ch\delta^2=1$ and $\ch\delta$ acts on $\ch H$ by $\ch
\theta$. Part of the data is  the  distinguished element $\ch\delta$
(more precisely its conjugacy class).

More generally an {\em E-group for $H$} is 
a group $\EH=\langle \ch H,\ch\delta\rangle$, where $\ch\delta$ acts on
$\ch H$ by $\ch\theta$, and $\ch\delta^2\in\ch H^{\ch\theta}$. 
Such a group is determined 
 up to isomorphism by the image of $\ch\delta^2$ in 
$\ch H^{\ch\theta}/\{h\ch\theta(h)\,|\,h\in \ch H\}$.  Again the data
includes the $\ch H$ conjugacy class of $\ch\delta$.    See
\cite{av1}*{Definition 5.9}. 

A homomorphism $\phi\colon W_\R\rightarrow \EH$ is said to
be {\em admissible} if it is continuous and $\phi(j)\in \LH\backslash
\ch H$. Conjugacy classes of admissible homomorphisms parametrize
genuine representations of 
$H(\R)_\gamma$.

\begin{lemma}[\cite{av1}*{Theorem 5.11}]
\label{l:egroup}
In the setting of Lemma 
\ref{l:lambdakappa}, suppose
$(1-\theta)\gamma\in X^*(H)$.
View $\gamma$ as an element of $\frac12 X_*(\ch H)$.
Let $\EH=\langle \ch H,\ch\delta\rangle$ where $\ch\delta$ acts on
$\ch H$ by $\ch\theta$, and  $\ch\delta^2=\exp(2\pi i\gamma)\in \ch H^{\ch\theta}$.

There is a canonical bijection between the irreducible genuine
characters of $H(\R)_\gamma$ and $\ch H$-conjugacy classes of
admissible homomorphisms $\phi\colon W_\R\rightarrow \EH$. 
\end{lemma}
If $\EH$ is the L-group of $H$ then (by restricting to $H(\R)\subset
H(\R)_\gamma$) we can replace genuine characters of $H(\R)_\gamma$
with characters of $H(\R)$.

\begin{proof}[Sketch of proof]
An admissible homomorphism $\phi$ may be written in the form
\begin{equation}
\label{e:phi1}
\begin{aligned}
\phi(z)&=z^\lambda\overline z^{\ch\theta(\lambda)}\\
\phi(j)&=\exp(2\pi i\mu)\ch\delta
\end{aligned}
\end{equation}
for some $\lambda,\mu\in \ch\h$. Then $\phi(j)^2=\exp(2\pi
i(\mu+\ch\theta\mu)+\gamma)$ and $\phi(-1)=\exp(\pi
i(\lambda-\ch\theta\lambda))$, so $\phi(j^2)=\phi(j)^2$ if and
only if
\begin{equation}
\label{e:kappacondition}
\kappa:=\frac12(1-\ch\theta)\lambda-(1+\ch\theta)\mu\in \gamma+ X_*(\ch H)=\gamma+X^*(H).
\end{equation}
In this case $(1+\theta)\lambda=(1+\theta)\kappa$; take $\phi$ to $\chi(\lambda,\kappa)$. 
\end{proof}

Write $\chi(\phi)$ for the genuine character of $H(\R)_\gamma$ associated to $\phi$. 

\bigskip

The Chevalley involution $C$ of $\ch H$ (i.e., inversion) extends to
an involution of 
$\E H=\langle \ch H,\ch\delta\rangle$, fixing $\ch\delta$ (this uses the 
fact that $\exp(2\pi i(2\gamma))=1$). 

Here is the main result in the case of (covers of) tori.

\begin{lemma}
\label{l:Ctorus}
Suppose $\phi\colon W_\R\rightarrow\EH$ is an admissible homomorphism, with
corresponding genuine character $\chi(\phi)$ of $H(\R)_\gamma$. 
Then
\begin{equation}
\chi(C\circ\phi)=\chi(\phi)^*
\end{equation}
\end{lemma}

\begin{proof}
Suppose $\phi$ is given by \eqref{e:phi1},
so $\chi(\phi)=\chi(\lambda,\kappa)$ with $\kappa$ as in
\eqref{e:kappacondition}. Then
\begin{equation}
\begin{aligned}
(C\circ\phi)(z)&=z^{-\lambda}\overline z^{\,-\ch\theta(\lambda)}\\
(C\circ\phi)(j)&=\exp(-2\pi i\mu)\ch\delta
\end{aligned}
\end{equation}
By \eqref{e:kappacondition}
$\chi(C\circ\phi)=\chi(-\lambda,-\kappa)=\chi(\lambda,\kappa)^*$.  
\end{proof}

\sec{L-packets without L-groups}
\label{s:lpackets}

Suppose $G$ is defined over $\R$, with real points $G(\R)$.  This
means that $G(\R)=G(\C)^\sigma$ where $\sigma$ is an antiholomorphic
involution. ({\em Antiholomorphic} means that if $f$ is a locally defined
holomorphic function on $G({\mathbb C})$, then 
$g\mapsto \overline{f(\sigma(g))}$ is also holomorphic.)
  Fix a {\em Cartan involution} $\theta$ of $G$ corresponding to
  $G(\R)$, and let $K=G^\theta$.  This means that $\theta$ and
  $\sigma$ commute, and $K(\R)=K\cap G(\R)=K^\sigma=G(\R)^\theta$ is a
maximal compact subgroup of $G(\R)$, with complexification $K$.  We
have the following picture, where each arrow represents taking the
fixed points with respect to the given involution.

$$
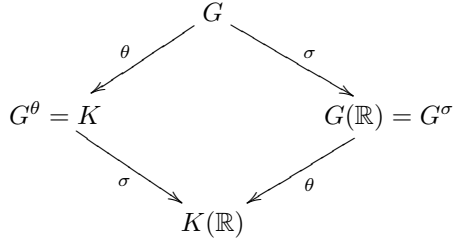
\xymatrix{
&G\ar[rd]^\sigma\ar[ld]_\theta\\
G^\theta=K\ar[rd]_\sigma&&G(\R)=G^\sigma\ar[ld]^\theta\\
&K(\R)
}
$$
We say that {\em $\theta$ corresponds to $\sigma$}, and vice-versa.

We work entirely in the algebraic setting. We consider
$(\g,K)$-modules, and write $(\pi,V)$ for a $(\g,K)$-module with
underlying complex vector space $V$. The set of equivalence classes of
irreducible $(\g,K)$-modules is a disjoint union of L-packets. In this
section we describe L-packets in terms of data for $G$ itself.  For
the relation with L-parameters see Section \ref{s:lparameters}. 

Suppose $H$ is a $\theta$-stable Cartan subgroup of $G$.
After conjugating  by $K$  we may assume it is defined over $\R$,
which we always do without further comment.

The {\em imaginary roots} $\Delta_i$, i.e., those fixed by $\theta$,
form a root system. 
Let $\rho_i$ be one-half the sum 
of a set $\Delta^+_i$ of positive imaginary roots. 
The two-fold cover $H_{\rho_i}$ of $H$ is defined as in Section
\ref{s:tori}. It is convenient to eliminate the dependence on
$\Delta^+_i$:
define $\wt H$
to be the inverse limit of $\{H_{\rho_i}\}$ over all choices of
$\Delta^+_i$. 
The inverse image of $H(\R)$ in
$H_{\rho_i}$ is denoted $H(\R)_{\rho_i}$, and take the inverse limit 
to define $\wt{H(\R)}$. (The existence and uniqueness of the maps
defining the inverse limit follows from standard facts about systems
of positive roots.)

\begin{definition}
An {\em L-datum} is a pair $(H,\chi)$ where $H$ is a  $\theta$-stable
Cartan subgroup  of $G$, 
$\chi$ is a genuine character of 
$\wt{H(\R)}$, and $\langle d\chi,\ch\alpha\rangle\ne0$
for all imaginary roots.  (The Cartan subgroup is included in this
notation for convenience, but it is implicit in the
character $\chi$; so we may speak of $\chi$ as an L-datum.
\end{definition}

Associated to each L-datum is an L-packet. 
We start by defining relative discrete series  L-packets.

We say $H(\R)$ is {\em relatively compact} if $H(\R)\cap G_d$ is
compact, where $G_d$ is the derived group of $G$. (It is equivalent to
require that all of the roots of $H$ in $G$ be imaginary.)  Then
$G(\R)$ has relative discrete series representations if and only if it
has a relatively compact Cartan subgroup.

Suppose $H(\R)$ is relatively compact. Choose 
a set of positive roots $\Delta^+$ and define the {\em Weyl denominator}
\begin{equation}
D(\Delta^+,h)= e^\rho(h) \prod_{\alpha\in
  \Delta^+}(1-e^{-\alpha}(h))\quad(h\in H(\R)_\rho).
\end{equation}
This is a genuine function, i.e.,
it satisfies $D(\Delta^+,\zeta
h)=-D(\Delta^+,h)$, and we view it as a function on $\wt{H(\R)}$.

Let $q=\frac12\dim(G_d/K\cap G_d)$.
Let $W(K,H)=\Norm_K(H)/H\cap K$; this is isomorphic to the real
Weyl group $W(G(\R),H(\R))=\Norm_{G(\R)}(H(\R))/H(\R)$.

\begin{definition}
\label{d:ds}
Suppose $(H,\chi)$ is an L-datum with  $H(\R)$  relatively
compact. 
Let $\pi=\pi(\chi)$ be the 
unique (up to equivalence) nonzero relative discrete series
representation whose character restricted to the
regular elements of $H(\R)$ is
\begin{equation}
\label{e:dschar}
\Theta_{\pi}(h)=(-1)^qD(\Delta^+,\wt h)\inv
\sum_{w\in W(K,H)}\sgn(w)(w\chi)(\wt h).
\end{equation}
Here $\wt h\in \wt{H(\R)}$ is any inverse image of $h$,
and $\Delta^+$ makes $d\chi$ dominant.
Every relative discrete series representation is obtained
this way, and $\pi(\chi)\simeq\pi(\chi')$ if and only if
$\chi$ and $\chi'$ are $W(K,H)$-conjugate.

The L-packet of $(H,\chi)$ is 
\begin{equation}
\label{e:piggamma}
\Pi_G(\chi)=\{\pi(w\chi)\,|\, w\in W(G,H)/W(K,H)\}.
\end{equation}
\end{definition}

It is a basic result of Harish-Chandra that $\pi(\chi)$ exists and
is characterized (among relative discrete series representations) by
this character formula. This version of the character formula is a
slight variant of the usual one, because of the use of $\wt{H(\R)}$. 
See \cite{av1} or \cite{characters}.

By \eqref{e:dschar} the representations in $\Pi_G(\chi)$ all have
infinitesimal character $d\chi$. 
If $\rho=\rho_i$ is one-half the sum of any choice of positive roots 
(all roots are imaginary), $\chi\otimes e^{\rho}$ factors to $H(\R)$, and
the central character of $\Pi_G(\chi)$ is 
$(\chi\otimes e^{\rho})|_{Z(G(\R))}$. 

Since $2\rho=2\rho_i$ is a sum of roots, $e^{2\rho}$ is trivial on the center $Z$ of $G$,
and there is a canonical splitting of the restriction of $\wt H$ to $Z$:
$z\rightarrow(z,1)\in H_{\rho}\simeq\wt H$.
Using this splitting the central character of the packet is simply
$\chi|_{Z(G(\R))}$. 

We are going to show that (as is well known) $\Pi(\chi)$ is precisely
the set of 
relative discrete series representations with the same infinitesimal
and central characters as $\pi(\chi)$. 
In fact something stronger is true.

Let $\Grad$ be the radical of
$G$. This maximal central torus is
the identity component of the center, and is defined over $\R$. 
By a {\em radical character}
we mean a character of $\Grad(\R)$, 
and the radical character of an irreducible representation is the restriction of its
central character to $\Grad(\R)$.

\begin{proposition}
\label{p:dslpacket}
An L-packet of relative discrete series representations is uniquely
determined by an infinitesimal  and a radical character.
\end{proposition}

What appears in the proof is in fact just the ``split radical
character,'' the character of the maximal split torus in the center of
$G$; more precisely, the character on the elements of order two in
that split torus.

The proof will be based on the following structural fact.

\begin{lemma}
\label{l:center}
Suppose $H(\R)$ is a relatively compact Cartan subgroup of $G(\R)$.
Then
\begin{equation}
\label{e:center}
\Grad(\R) \subset Z(G(\R))\subset\Grad(\R)H(\R)^0 \subset H(\R).
\end{equation}
\end{lemma}

\begin{proof} 
Let $A$ be a maximal split subtorus of $H$ (i.e., $\theta(a)=a\inv$
for all $a\in A$). Since $H(\R)$ is relatively compact, $A$ is
contained in the radical $\Grad$ of $G$. It is well known 
that $H(\R)=A(\R)H(\R)^0$. (Such a statement is true for reductive groups 
\cite{borel_tits}*{Theorem 14.4}, with $A$ a maximal split torus; for
tori it is elementary.)  
Therefore $H(\R)=A(\R)H(\R)^0=\Grad(\R)H(\R)^0$. Since
$Z(G(\R))\subset H(\R)$ this proves the lemma. 
\end{proof}

\begin{proof}[Proof of Proposition \ref{p:dslpacket}]
  Suppose $(H,\chi)$ is an L-datum as in Definition \ref{d:ds}, and
  $\Pi = \Pi_G(\chi)$ the corresponding L-packet of relative discrete
  series representations. Write $\lambda\in\Hom(\h,\C)$ for the
  differential of $\chi$. The infinitesimal character $\chiinf$ of the
  L-packet is has as representatives exactly all the weights
  $w\lambda$, with $w\in W(G,H)$.

Write $\chi=(\lambda,\kappa)$ as in Lemma \ref{l:lambdakappa}. The
L-packet of $\gamma$ by definition consists of the representations
$\pi(w\lambda,w\kappa)$ (with $w\in W(G,H)$). The set of {\em all}
relative discrete series of infinitesimal character $\chiinf$ is equal
to 
$$\{\pi(w\lambda,w\kappa + \kappa_1)\}.$$
Here the modification term $\kappa_1$ is (by Lemma
\ref{l:lambdakappa}) subject to the requirement $(1+\theta)\kappa_1 =
0$; that is,
$$\kappa_1\in X^*(H)^{-\theta}/(1-\theta)X^*(H).$$
The right side here is the group of characters of 
$$H^\theta/H^\theta_0 = A(\R)^\theta/(A(\R)\cap H^\theta_0).$$
We have therefore shown that every relative discrete series
representation $\pi_1$ of infinitesimal character $\lambda$ arises by
{\em changing} the radical character of some $\pi(w\chi)$ by
$\kappa_1$. If the radical character is unchanged---that is, if
$\pi(\chi)$ and $\pi_1$ have the same radical character---then
$\kappa_1$ belongs to $(1-\theta)X^*(H)$, and $\pi_1 =\pi(w\chi)$
belongs to the L-packet of $\pi$.
\end{proof}

The following converse to Proposition \ref{p:dslpacket} follows immediately from 
the definitions.
\begin{lemma}
\label{l:nonzerolpacket}
Assume $G(\R)$ has a relatively compact Cartan subgroup, 
and fix one, denoted $H(\R)$.
Suppose $\chiinf,\chirad$ are infinitesimal and radical characters,
respectively. Choose $\lambda\in\Hom(\h,\C)$ defining $\chiinf$ via the Harish-Chandra
homomorphism.
Then the L-packet of relative discrete series representations defined
by $\chiinf,\chirad$ is nonzero if and only if $\lambda$ is regular, and there is a genuine
character of $\wt{H(\R)}$ satisfying:
\begin{enumerate}
\item[(1)] $d\chi=\lambda$
\item[(2)] $\chi|_{\Grad(\R)}=\chirad$.
\end{enumerate}
The conditions are independent of the choices of $H(\R)$ and $\lambda$.
\end{lemma}
In (2) we have used the splitting $Z(G(\R))\rightarrow\wt{H(\R)}$
discussed after \eqref{e:piggamma}. 

\bigskip

We now describe general L-packets.
See \cite{characters}*{Section 13} or \cite{unitaryDual}*{Section 6}.

\begin{definition}
\label{d:lpacket}
Suppose $(H,\chi)$ is an L-datum. 
Let $A$ be the identity component of $\{h\in H\,|\,
\theta(h)=h\inv\}$ and set $M=\Cent_G(A)$.
Let $\mathfrak a=\Lie(A)$. 
Choose a parabolic subgroup $P=MN$ satisfying 
$$
\RE\langle d\chi|_\mathfrak a,\ch\alpha\rangle\ge0\text{ for
  all roots of }\h\text{ in }\Lie(N).
$$
Then $P$ is defined over $\R$,  and 
$H(\R)$ is a relatively compact Cartan subgroup of $M(\R)$.
Let $\Pi_M(\chi)$ be the L-packet of relative discrete series
representations of $M(\R)$ as in Definition \ref{d:ds}.
Define
\begin{equation}
\Pi_G(\gamma)=\bigcup_{\pi\in\Pi_M(\chi)}
\{\text{irreducible quotients of }\Ind_{P(\R)}^{G(\R)}(\pi)\}
\end{equation}
Here we use normalized induction, and  pull $\pi$ back to $P(\R)$ via
the map $P(\R)\rightarrow M(\R)$ as 
usual. 
\end{definition}

By the discussion following Definition \ref{d:ds}, and
basic properties of induction, 
the infinitesimal character of $\Pi_G(\chi)$ is $d\chi$, 
and the central character is $\chi|_{Z(G(\R))}$.

\sec{The Tits Group}
\label{s:tits}
We need a few structural facts provided by the Tits group.

Fix a pinning $\caP=(B,H,\{X_\alpha\})$ (see Section \ref{s:aut}).
For $\alpha\in\Pi$ 
define $X_{-\alpha}\in \g_{-\alpha}$ by 
$[X_\alpha,X_{-\alpha}]=\ch\alpha$ as in Section \ref{s:aut}. 
Define $\sigma_\alpha\in W=\exp(\frac\pi 2(X_\alpha-X_{-\alpha}))\in\Norm_G(H)$.
The image of $\sigma_\alpha$ in $W=W(G,H)$ is the simple reflection $s_\alpha$.
Let $H_2=\{h\in H\,|\, h^2=1\}$.

\begin{definition}
\label{d:tits}
The {\em Tits group defined by $\caP$} is the subgroup $\T$of $G$
generated by $H_2$ and
$\{\sigma_\alpha\,|\,\alpha\in\Pi\}$.   
\end{definition}

\begin{proposition}[\cite{tits_group}]
\label{p:titsrels}
The Tits group $\caT$ has the given generators, and relations:

\begin{enumerate}
\item[(1)] $\sigma_\alpha h\sigma_\alpha\inv=s_\alpha(h)$,
\item[(2)] the braid relations among the $\sigma_\alpha$,
\item[(3)] $\sigma_\alpha^2=\ch\alpha(-1)$.
\end{enumerate}
If $w\in W$ then there is a canonical representative $\sigma_w$ of $w$
in $\caT$ 
defined as follows. Suppose $w=s_{\alpha_1}\dots s_{\alpha_n}$ is  a 
reduced expression with each $\alpha_i\in \Pi$. Then
$\sigma_w=\sigma_{\alpha_1}\dots \sigma_{\alpha_n}$, independent of
the choice of reduced expression.
\end{proposition}

\begin{lemma}
\label{l:longfixed}
If $w_0$ is the long element of the Weyl group, then
$\sigma_{w_0}$ is fixed by any $\caP$-distinguished automorphism.
\end{lemma}

\begin{proof}
Suppose $w_0=s_{\alpha_1}\dots s_{\alpha_n}$ is a reduced expression.
If $\gamma\in\Aut(\caP)$ it
induces an automorphism of the Dynkin diagram, 
so $s_{\gamma(\alpha_1)}\dots s_{\gamma(\alpha_n)}$ is also a reduced
expression for $w_0$. 
Therefore $\gamma(\sigma_{w_0})=
\gamma(\sigma_{\alpha_1}\dots\sigma_{\alpha_n})=
\sigma_{\gamma(\alpha_1)}\dots\sigma_{\gamma(\alpha_n)}=\sigma_{w_0}$
by the last assertion of Proposition \ref{p:titsrels}.
\end{proof}

Let $\ch\rho$ be one-half the sum of the positive coroots.

\begin{lemma}
\label{l:titslemma}
For any $w\in W$ we have
\begin{equation}
\label{e:titslemma}
\sigma_w\sigma_{w\inv}=\exp(\pi i(\ch\rho-w\ch\rho)).
\end{equation}
In particular if $w_0$ is the long element of the Weyl group,
\begin{equation}
\label{e:w0}
\sigma_{w_0}^2=\exp(2\pi i\ch\rho)\in Z(G).
\end{equation}
This element of $Z(G)$ is independent of the choice of positive roots, and is fixed by 
every automorphism of $G$.
\end{lemma}

\begin{proof}
We proceed by induction on the length of $w$.
If $w$ is a simple reflection $s_\alpha$ then
$s_\alpha\ch\rho=\ch\rho-\ch\alpha$, and this reduces to Proposition
\ref{p:titsrels}(3).

Write $w=s_\alpha u$ with $\alpha$ simple and $\ell(w)=\ell(u)+1$. 
Then $\sigma_{w}=\sigma_\alpha\sigma_u$, $w\inv=u\inv s_\alpha$, $(\sigma_w)\inv=\sigma_{u\inv}\sigma_\alpha$, 
and
\begin{equation}
\begin{aligned}
\sigma_w\sigma_{w\inv}&=
\sigma_\alpha\sigma_u\sigma_{u\inv}\sigma_\alpha\\
&=
\sigma_\alpha\sigma_u\sigma_{u\inv}\sigma_\alpha^{-1}\exp(\pi i\ch\alpha)\\
&=
\sigma_\alpha\exp(\pi i(\ch\rho-u\ch\rho))\sigma_\alpha^{-1}\exp(\pi i\ch\alpha)\quad\text{
  (by the inductive step)}\\
&=
\exp(\pi i(s_\alpha\ch\rho-w\ch\rho))\exp(\pi i\ch\alpha)\\
&=
\exp(\pi i(\ch\rho-\ch\alpha-w\ch\rho))\exp(\pi i\ch\alpha)=
\exp(\pi i(\ch\rho-w\ch\rho)).
\end{aligned}
\end{equation}
The final assertion is easy.
\end{proof}

We thank Marc van Leeuwen for this proof.

We only need what follows for the part of the main theorem involving
the Chevalley involution $\tau$ of $W_\R$. 
Let $C=C_{\caP}$ be the Chevalley involution defined by  $\caP$.

\begin{lemma}
$C(\sigma_w)=(\sigma_{w\inv})^{-1}$
\end{lemma}

\begin{proof}
We proceed by induction on the length of $w$.

Since
$C(X_\alpha)=X_{-\alpha}$ ($\alpha\in\Pi$), we conclude
$C(\sigma_\alpha)=\sigma_\alpha\inv$. 

Suppose $w=s_\alpha u$ with $\text{length}(w)=\text{length}(u)+1$. 
Then $\sigma_w=\sigma_\alpha \sigma_{u}$, and
$C(\sigma_w)=C(\sigma_\alpha)C(\sigma_{u})=
\sigma_\alpha\inv(\sigma_{u\inv})\inv$.
On the other hand $w\inv=u\inv s_\alpha$, so
$\sigma_{w\inv}=\sigma_{u\inv}\sigma_\alpha$, and taking the
inverse gives
the result.
\end{proof}

Fix a $\caP$-distinguished involution $\tau$ of $G$.
Consider the semidirect product $G\rtimes\langle\delta\rangle$
where $\delta^2=1$ and $\delta$ acts on $G$ by $\tau$.
By  Lemma \ref{l:commutes}, $C=C_\caP$
extends to the semidirect product, fixing $\delta$.
Since $\tau$ normalizes $H$
 and $B$, it defines an automorphism of $W$, 
satisfying $\tau(\sigma_w)=\sigma_{\tau(w)}$.
\begin{lemma}
\label{l:tits}
Suppose $w\in W$ satisfies $w\tau(w)=1$. Then

\smallskip

\noindent (a) $C(\sigma_w\delta)=(\sigma_w\delta)\inv$

\smallskip

\noindent (b) Suppose $g\in\Norm_G(H)$ is a representative
of $w$. Then $C(g\delta)$ is $H$-conjugate to $(g\delta)\inv$. 
\end{lemma}

\begin{proof}
By the previous lemma, and using $\tau(w)=w\inv$, we compute
$$
C(\sigma_w\delta)\sigma_w\delta
=
(\sigma_{w\inv})\inv\sigma_{\tau(w)}=
(\sigma_{\tau(w)})\inv\sigma_{\tau(w)}=1.
$$
For (b) write $g=h\sigma_w\delta$ with $h\in H$. 
Then 
$C(g)=C(h\sigma_w\delta)=h\inv(\sigma_w\delta)\inv=h\inv(h\sigma_w\delta)\inv
h$. 
\end{proof}

\sec{L-parameters}
\label{s:lparameters}

Fix a Cartan involution $\theta$ of $G$.
Let $\ch G$ be the connected, complex dual group of $G$. 
The
L-group $\LG$ of $G$ is $\langle \ch G,\ch\delta\rangle$ where
$\ch\delta^2=1$, and $\ch\delta$ acts on $\ch G$ by a homomorphism
$\ch\theta_0$, which we now describe.
See \cite{borelCorvallis}, \cite{av1}, or  \cite{algorithms}*{Section 2}.

Fix Borel and Cartan subgroups $B_0,H_0$ and let 
$$D_b=(X^*(H_0),\Pi,X_*(H_0),\ch\Pi)$$
be the corresponding based root datum. Similarly 
choose $\ch B_0,\ch H_0$ for $\ch G$ to define $\ch D_b$. 
We identify $X^*(H_0)=X_*(\ch H_0)$ and 
$X_*(H_0)=X^*(\ch H_0)$. 
Also fix a pinning $\ch\caP=(\ch B_0,\ch H_0,\{X_{\ch\alpha}\})$ for
$\ch G$.
See  Section \ref{s:aut}.

An automorphism $\mu$ of $D_b$ consists of a pair
$$(\tau,\tau^t) \in \Aut(X^*(H_0))\times \Aut(X_*(H_0)),$$ 
where the transpose is defined with respect to the 
perfect pairing 
$$X^*(H_0)\times X_*(H_0)\rightarrow \Z.$$
By definition $\tau$ and $\tau^t$  preserve $\Pi,\ch\Pi$ respectively. 
Interchanging $(\tau,\tau^t)$ defines a transpose isomorphism
$\Aut(D_b)\simeq \Aut(\ch D_b)$, denoted $\mu\rightarrow\mu^t$. 
Compose with the embedding 
$\Aut(\ch D_b)\hookrightarrow\Aut(\ch G)$ defined by $\ch\caP$
(Section \ref{s:aut}) to define a map:
\begin{equation}
\label{e:transpose}
\mu\rightarrow\mu^t\colon \Aut(D_b)\hookrightarrow \Aut(\ch G).
\end{equation}

Suppose $\sigma$ is a real form corresponding to
$\theta$
(see the beginning of Section \ref{s:lpackets}). 
Choose $g\in G$ conjugating $\sigma(B_0)$ to $B_0$ and $\sigma(H_0)$ to
$H_0$.  Then $\tau=\iNT(g)\circ\sigma\in \Aut(D_b)$.
Let $\ch\theta_0=\tau^t\in\Aut(\ch G)$. See \cite{borelCorvallis}.
For example, if $G(\R)$ is split, taking $B_0,H_0$ defined over $\R$
shows that $\LG=\ch G\times\Gamma$ (direct product).

Alternatively, using $\theta$ itself gives a version 
naturally related to the most compact Cartan subgroup. 
Let $\gamma$ be the image of $\theta$ in $\Out(G)\simeq
\Aut(D_b)$.
The automorphism $-w_0$  of $X^*(H)$, taking $\gamma$ to $-(w_0(\gamma))$, induces
an automorphism of $D_b$, also denoted $-w_0$.
Thus $-w_0\gamma\in\Aut(D_b)$, and we may define:
\begin{equation}
\label{e:chtheta0}
\ch\theta_0=(-w_0\gamma)^t\in \Aut(\ch G).
\end{equation}
This is the approach of \cite{av1} and \cite{algorithms}. 
It is not hard to see the elements $\tau,\gamma\in\Aut(D_b)$ satisfy
$\tau=-w_0\gamma$, so the two definitions of $\LG$ agree.

\begin{lemma}
\label{l:cpt}
The following conditions are equivalent:
\begin{enumerate}
\item[(1)] $G(\R)$ has a compact Cartan subgroup,
\item[(2)] $\ch\theta_0$ is inner to the Chevalley involution;
\item[(3)] There is an element $y\in\LG\backslash\ch G$ such that 
$yhy\inv=h\inv$ for all $y\in \ch H_0$.
\end{enumerate}
\end{lemma}

\begin{proof}
The equivalence of (2) and (3) is immediate. 
Let $C$ be the Chevalley involution of $\ch G$
with respect to $\ch\caP$. 
It is easy to see \eqref{e:chtheta0} is equivalent to:
the image of $\ch\theta_0\circ C$ in $\Out(\ch G)\simeq\Aut(\ch D_b)$ is  equal to
$\gamma^t$. 
So the assertion is that
$G(\R)$ has a compact Cartan subgroup if and only if $\gamma=1$,
i.e., $\theta$ is an inner automorphism. The direction $\Rightarrow$ holds  by Lemma \ref{l:inth}.
For the other direction, if $\theta=\iNT(x)$ for $x\in G(\C)$,
let $H$ be any Cartan subgroup of $G$ containing $x$. Then $\theta_x$ acts trivially on $H$, i.e., 
$H$ is the complexification of a compact Cartan subgroup of $G(\R)$.
\end{proof}

A homomorphism  $\phi\colon W_\R\rightarrow \LG$ is said to be {\it quasiadmissible} if it is
continuous, $\phi(\C^*)$ consists of semisimple elements, and
$\phi(j)\in \LG\backslash \ch G$
\cite{borelCorvallis}*{8.2}.
We will see in Sections \ref{s:relative} and \ref{s:general} that every
quasiadmissible homomorphism is associated to an L-packet
$\Pi_G(\phi)$, 
which depends only on the $\ch G$-conjugacy class of $\phi$. 
We say $\phi$ is {\it admissible} if it satisfies the additional 
relevancy condition \cite{borelCorvallis}*{8.2(ii)}.
The admissible condition (unlike quasiadmissibility) is sensitive to the real forms of $G$,
and  guarantees that $\Pi_G(\phi)$ is nonempty. If $G(\R)$ is quasisplit every 
quasiadmissible homomorphism is admissible.

After conjugating by $\ch G$ we may assume $\phi(\C^*)\subset \ch
H_0$. 
Let $\ch S$ be the centralizer of $\phi(\C^*)$ in $\ch G$.
Since $\phi(\C^*)$ is connected, abelian and consists of semisimple
elements, $\ch S$ is a  connected reductive complex group, and $\ch
H_0$ is a Cartan subgroup of $\ch S$.
Conjugation by $\phi(j)$ is an
involution of $\ch S$, so $\phi(j)$ normalizes a
Cartan subgroup of $\ch S$.
Equivalently some $\ch S$-conjugate of $\phi(j)$ normalizes $\ch H_0$;
after this change we 
may assume $\phi(W_\R)\subset \Norm_{\ch G}(\ch H_0)$.

Therefore 
\begin{subequations}
\renewcommand{\theequation}{\theparentequation)(\alph{equation}}  
\label{e:phi}
\begin{align}
\label{e:phia}
\phi(z)&=z^\lambda\overline z^{\lambda'}
\quad(\text{for some }\lambda,\lambda'\in X_*(\ch H_0)\otimes\C,\,\,\lambda-\lambda'\in
X_*(\ch H_0))\\
\label{e:phib}
\phi(j)&=h\sigma_w\ch\delta\quad(\text{for some }w\in W,h\in \ch H_0).
\end{align}
Here
(a) is shorthand for
$\phi(e^s)=\exp(s\lambda+\overline s\lambda')\in \ch H_0$ ($s\in\C$), and
the condition on $\lambda-\lambda'$ guarantees this is
well defined. In (b) we're using the element $\sigma_w$ of the Tits
group representing $w$ (Proposition \ref{p:titsrels}).

Conversely, given $\lambda,\lambda',w$ and $h$,
(a) and (b) give
a well-defined homomorphism
$\phi\colon W_\R\rightarrow\LG$ 
if and only if
\begin{align}
&\ch\theta:=\iNT(h\sigma_w\ch\delta)\text{ is an involution of }\ch H_0,\\
&\lambda'=\ch\theta(\lambda),\\
&h\ch\theta(h)(\sigma_w\ch\theta_0(\sigma_w))=\exp(\pi i(\lambda-\ch\theta(\lambda)).
\end{align}

Furthermore (c) is equivalent to 
\begin{equation}
\tag{\ref{e:phi})(c$'$}w\ch\delta(w)=1.
\end{equation}
\end{subequations}

\subsec{Infinitesimal and radical characters}
\label{s:chars}

Suppose $\phi$ is as in \eqref{e:phi}.
We attach two invariants to an admissible homomorphism $\phi$.

\noindent{\bf Infinitesimal Character of $\phi$}

View $\lambda$ as an element of $X^*(H_0)\otimes\C$ via the 
identification $X_*(\ch H_0)=X^*(H_0)$.
The $W(G,H_0)$-orbit of $\lambda$ is independent of all choices, so 
it defines an infinitesimal character  for $G$,
denoted $\chiinf(\phi)$, via the
Harish-Chandra homomorphism. 
\medskip

\noindent{\bf  Radical character of $\phi$}

Recall (Section \ref{s:lpackets}) $\Grad$ is the radical of $G$, and the radical
character of a representation is its restriction to $\Grad(\R)$.

Dual to the inclusion $\iota\colon \Grad\hookrightarrow G$ is a
surjection $\ch\iota\colon \ch
G\twoheadrightarrow \ch[\Grad]$.
For an L-group for $\Grad$ we can take
 $\L\Grad=\langle \ch
[\Grad],\ch\delta\rangle$. 
Thus $\ch\iota$ extends to   a natural surjection
$\ch\iota\colon \LG\rightarrow\L\Grad$ (taking $\ch\delta$ to itself). 
Then $\ch\iota\circ\phi\colon W_\R\rightarrow\L\Grad$, and this defines a
character of $\Grad(\R)$ by the construction of Section
\ref{s:tori}.
We denote this character $\chi_{rad}(\phi)$. 
See \cite{borelCorvallis}*{10.1} and \cite{langlandsClassification}*{page 20}.

\subsec{Relative Discrete Series L-packets}
\label{s:relative}

By a {\em Levi subgroup} of $\LG$ we mean the centralizer $\d M$ of a torus
$\ch T\subset \ch G$, which meets both components of $\LG$
\cite{borelCorvallis}*{Lemma 3.5}. 
We will soon see that an L-packet $\Pi_G(\phi)$ consists of relative discrete series representations
if and only if $\phi(W_\R)$ is not contained in a proper Levi
subgroup.

\begin{lemma}
\label{l:notcontained}
Suppose $\phi$ is as in \eqref{e:phi}.
If $\phi(W_\R)$ is not contained in a proper Levi subgroup then
$\lambda$ is regular and $G(\R)$ has a relatively compact Cartan
subgroup.   
\end{lemma}

See \cite{borelCorvallis}*{Lemma 11.1} and
\cite{langlandsClassification}*{Lemmas 3.1 and 3.3}. 

\begin{proof}
Assume $\phi(W_\R)$ is not contained in a proper Levi subgroup. 
Let $\ch S=\Cent_{\ch G}(\phi(\C^*))$  as in the discussion preceding
\eqref{e:phi}.
Then $\ch\theta=\iNT(\phi(j))$ is an   involution of $\ch S$, and of
its derived group $\ch S_d$. 
There cannot be a torus in $\ch S_d$, fixed (pointwise) by $\ch\theta$; its
centralizer would contradict the assumption. 
Since any involution of a semisimple group fixes a
torus, this implies $\ch S_d=1$, i.e., $\ch S=\ch H_0$, 
which implies $\lambda$ is regular.

Similarly, there can be no torus in $\ch H_0\cap\ch G_d$ fixed
by $\ch\theta$. This implies $\ch\theta(h)=h\inv$ for all $h\in \ch
H_0\cap\ch G_d$. By Lemma \ref{l:cpt} applied to the derived group,
$G(\R)$ has a relatively compact Cartan subgroup. 
\end{proof}

\begin{definition}
\label{d:dslpacket}
In the setting of Lemma \ref{l:notcontained},
$\Pi_G(\phi)$ is the L-packet of relative discrete series
representations determined by infinitesimal character $\chiinf(\phi)$ 
and radical character $\chirad(\phi)$ (see Proposition  \ref{p:dslpacket}).
\end{definition}

\begin{lemma}
$\Pi_G(\phi)$ is nonempty.
\end{lemma}

\begin{proof}
After conjugating $\theta$ if necessary we may assume $H_0$ is
$\theta$-stable. 
Write $\phi$ as in \eqref{e:phi}(a) and (b).
By the discussion of infinitesimal character above $\chiinf(\phi)$ is
defined by $\lambda$, viewed as an element of 
$X^*(H_0)\otimes\C$.

Choose positive roots making $\lambda$ dominant. 
By Lemma \ref{l:nonzerolpacket} it is enough to construct a genuine
character of $H(\R)_\rho$ satisfying $d\Lambda=\lambda$ and
$\Lambda|_{\Grad(\R)}=\chirad$. 
For this we apply Lemma \ref{l:egroup}, 
using the fact that
$\phi(W_\R)\subset \langle \ch
H_0,\sigma_w\ch\delta\rangle$. 
We being by identifying
$\langle \ch H_0,\sigma_w\ch\delta\rangle$ as an E-group.

First we claim $w=w_0$.
By \eqref{e:phi}(b) and (c) $\ch\theta|_{\ch H_0}=w\ch\theta_0|_{\ch H_0}$.
By \eqref{e:chtheta0}, for $h\in \ch H_0\cap\ch G_d$ we have: 
\begin{equation}
\ch\theta(h)=w\ch\theta_0(h)=w(-w_0\gamma^t)(h)=ww_0\gamma^t(h)\inv.
\end{equation}
Since $G_d(\R)$ has a compact Cartan subgroup, $\gamma$ is trivial on $H_0\cap
G_d$ and $\gamma^t(h)=h$.
On the other hand, as in the proof  of Lemma \ref{l:notcontained}, $\ch\theta(h)=h\inv$.
Therefore $ww_0=1$, i.e., $w=w_0$.

Next we compute
\begin{equation}
\label{e:sigma^2}
\begin{aligned}
(\sigma_{w_0}\ch\delta)^2&=w_0\theta_0(w_0)\\
&=w_0^2\quad(\text{by }\eqref{l:longfixed}\text{, since
}\ch\theta_0\text{ is distinguished})\\
&=\exp(2\pi i\rho)\quad\text{ (by }\eqref{e:w0}).
\end{aligned}
\end{equation}

Thus, in the terminology of Section \ref{s:tori},
$\langle \ch H_0,\sigma_{w_0}\ch\delta\rangle$ is identified with 
the E-group for $H$ defined by $\ch\rho$. 
Consequently $\phi\colon W_\R\rightarrow\langle\ch H_0,\sigma_{w_0}\ch\delta\rangle$
defines a genuine character $\chi$ of $H(\R)_\rho$. 

By construction $d\chi=\lambda$. The fact that
$\chi|_{\Grad(\R)}=\chirad$ is a straightforward check. Here are
the details.

Write $h$ of \eqref{e:phi}(b) as $h=\exp(2\pi i\mu)$.  We use the
notation of Section \ref{s:tori}, especially \eqref{e:phi1}.
Using the fact that
$\sigma_{w_0}\ch\delta$ is the distinguished element of the E-group of
$H_0$ we have
$\chi=\chi(\lambda,\kappa)$
where
\begin{equation}
\label{e:kappa}
\kappa=\frac12(1-\ch\theta)\lambda-(1+\ch\theta)\mu\in \rho+X^*(H_0).
\end{equation}
Write 
$p\colon X^*(H_0)\rightarrow X^*(\Grad)$ 
for the map dual to inclusion $\Grad\rightarrow H_0$.
Recall (Section \ref{s:lpackets}) there is a canonical splitting of 
the cover of $Z(G(\R))$; using this splitting $\chi(\rho,\rho)|_{Z(G(\R))}=1$, 
and 
$$
\chi|_{\Grad(\R)}=\chi(p(\lambda),p(\kappa-\rho)).
$$

On the other hand, by the discussion of the character of $\Grad(\R)$
above,
the E-group of $\Grad$ is $\langle \ch\Grad,\ch\delta\rangle$.
The map $p\colon X^*(H_0)\rightarrow X^*(\Grad)$ is identified with a 
map $p\colon X_*(\ch H_0)\rightarrow X_*(\ch\Grad)$.
Then $\ch\iota(\phi(j))=\ch\iota(h\sigma_w\ch\delta)=\ch\iota(h)\ch\delta=\exp(2\pi
ip(\mu))\ch\delta$. 
Let 
$$
\kappa'=\frac12(1-\ch\theta)p(\lambda)-(1+\ch\theta)p(\mu)\in X^*(\Grad)
$$
Thus $\kappa'=p(\kappa-\rho)$.  By the construction of Section
\ref{s:tori} applied to $\phi\colon W_\R\rightarrow \ch\Grad$,
$\chirad=\chi(p(\lambda),\kappa')=\chi(p(\lambda),p(\kappa-\rho))=\chi|_{\Grad(\R)}$.
\end{proof}

\begin{remark}
The fact that $(\sigma_{w_0}\ch\delta)^2=\exp(2\pi i\rho)$ is 
the analogue of \cite{langlandsClassification}*{Lemma 3.2}.
\end{remark}

We can read off the central character of the L-packet from the
construction. We defer this until we consider general L-packets
(Lemma \ref{l:charforchi}).
 
\bigskip

\subsec{General L-packets}
\label{s:general}

See \cite{borelCorvallis}*{Section 11.3} and
\cite{langlandsClassification}*{pages 40--58}. 

Recall 
(see the beginning of the previous section)
a Levi subgroup $\d M$ of $\LG$ is the centralizer of a torus $\ch T$, which meets both
components of $\LG$.
An admissible homomorphism $\phi$ may factor through various Levi subgroups $\d M$. 
We first choose $\d M$ so that $\phi\colon W_\R\rightarrow\d M$ defines a
relative discrete series L-packet of $M$.

Choose
a maximal torus $\ch T\subset \Cent_{\ch G}(\phi(W_\R))$
and define
\begin{equation}
\label{e:levi}
\ch M=\Cent_{\ch G}(\ch T),\,\d M=\Cent_{\LG}(\ch T).
\end{equation}
Then $\d M=\langle \ch M,\phi(j)\rangle$, so $\d M$ is a Levi subgroup,
and $\phi(W_\R)\subset\d M$.

Suppose $\phi(W_\R)\subset\Cent_{\d M}(\ch U)$
where $\ch U\subset\ch M$ is a torus.
Then $\ch U$ centralizes $\phi(W_\R)$ and  $\ch T$,
so $\ch U\ch T$ is a torus in $\Cent_{\ch G}(\phi(W_\R))$.
By maximality  $\ch U\subset\ch T$ and $\Cent_{\d M}(\ch U)=\d M$. 
Therefore
$\phi(W_\R)$ is not contained in any proper Levi subgroup of $\d M$.

\begin{lemma}
\label{l:independent}
The group $\d M$ is independent of the choice of $\ch T$, up to
conjugation by $\Cent_{\ch G}(\phi(W_\R))$.  
\end{lemma}

\begin{proof}
Conjugation by $\phi(j)$ defines an involution of the connected reductive group
$\ch S=\Cent_{\ch G}(\phi(\C^*))$
(see the discussion after Lemma \ref{l:cpt}),
and $\Cent_{\ch G}(\phi(W_\R))$ is the fixed points of this involution. 
Thus $\ch T$ is a maximal torus in (the identity component of) this reductive
group, and any two such tori are conjugate by $\Cent_{\ch G}(\phi(W_\R))$. 
\end{proof}

The idea is to identify 
$\d M$ with the L-group of
a Levi subgroup $M'(\R)$ of $G(\R)$. Then,
since $\phi(W_\R)\subset \d M$ is not 
contained in any proper Levi subgroup,
it defines  a relative discrete series L-packet for
$M'(\R)$. We obtain $\Pi_G$ by induction. Here are the details.

We need to  identify $\d M=\langle \ch
M,\phi(j)\rangle$ with an L-group. A crucial technical point is
that after conjugating we may assume $\d M=\langle \ch
M,\ch\delta\rangle$, making this identification clear.

\begin{lemma}[\cite{borelCorvallis}, Section 3.1]
\label{l:standardM}
Suppose $S$ is a $\ch\theta_0$-stable subset of $\ch\Pi$.
Let $\ch M_S$ be the corresponding Levi subgroup of $\ch G$:
$\ch H_0\subset\ch M_S$, and $S$ is a set of simple roots
of $\ch H_0$ in $\ch M_S$. Let $\d M_S=\ch
M_S\rtimes\langle\ch\delta\rangle$, a Levi subgroup of $\LG$.

Let $M_S\supset H_0$ be the Levi subgroup of $G$ with simple roots
$\{\alpha\,|\,\ch\alpha\in S\}\subset\Pi$.
Suppose some conjugate  $M'$ of $M_S$ is defined over $\R$. 
Write $\L M'=\ch M'\rtimes\langle\ch\delta_{M'}\rangle$.
Then conjugation  induces an isomorphism $\L M'\simeq \d M_S$,
taking $\ch\delta_{M'}$ to $\ch\delta$.

Any Levi subgroup of $\L G$ is $\ch G$-conjugate to $\d M_S$ for some
$\ch\theta_0$-stable set $S$.
\end{lemma}
We refer to the Levi subgroups $\d M_S$ of the lemma 
(where $S$ is $\ch\theta_0$-stable)  as {\em standard Levi subgroups}.

\begin{definition}
\label{d:general}
Suppose $\phi\colon W_\R\rightarrow \LG$ is an admissible homomorphism.
Choose a maximal torus $\ch T$ in $\Cent_{\ch G}(\phi(W_\R))$,
and define $\ch M,\d M$ by \eqref{e:levi}.

After conjugating by $\ch G$, we 
we may assume $\d M$ is a standard Levi subgroup.  Let $M(\C)$
be the corresponding standard Levi subgroup of $G(\C)$.

Assume there is a subgroup $M'$ conjugate to $M$, which is defined
over $\R$; otherwise define $\Pi(\phi)$ to be empty.
Let $\Pi_{M'}(\phi)$ be the L-packet for $M'(\R)$ defined  by
$\phi\colon W_\R\rightarrow \d M\simeq \L M'$. 
(cf. Lemma \ref{l:standardM}).
Define the {\em L-packet for $G$ attached to $\phi$}
$\Pi_G(\phi)$ by induction from $\Pi_{M'}(\phi)$ as in Definition
\ref{d:lpacket}.  

\end{definition}

\begin{lemma}
The L-packet $\Pi_G(\phi)$ is independent of all choices. 
\end{lemma}

\begin{proof}
By Lemma \ref{l:independent} the choice of $\ch T$ is irrelevant:
another choice leads to an automorphism of $\d M$ fixing $\phi(W_\R)$
pointwise. 

It is straighforward to see that the other choices, including
another Levi subgroup $M''$,  would give an element $g\in G(\C)$ such that
$\iNT(g)\colon M'\rightarrow M''$ is defined over $\R$, and
this isomorphism takes $\Pi_{M'}(\phi)$ to $\Pi_{M''}(\phi)$.

First of all we claim $M'(\R)$ and $M''(\R)$ are $G(\R)$-conjugate.
To see this, let $H'(\R)$ be a relatively compact Cartan subgroup of
$M'(\R)$. 
Then $H''(\R)=gH'(\R)g\inv$ is a relatively compact Cartan subgroup of $M''(\R)$. 
In fact $H'(\R)$ and $H''(\R)$ are $G(\R)$-conjugate: 
two Cartan subgroups of $G(\R)$ are $G(\C)$-conjugate if and
only if they are $G(\R)$-conjugate. Therefore $M'(\R),M''(\R)$, 
being the
centralizers of the split components of $H'(\R)$ and $H''(\R)$,
are also $G(\R)$-conjugate.

Therefore, since the inductive step is not affected by conjugating by $G(\R)$, we
may assume $M'=M''$. Then $g\in\Norm_{G(\C)}(M')$, and furthermore
$g\in\Norm_{G(\C)}(M'(\R))$.

Now $gH'(\R)g\inv$ is another relatively compact Cartan subgroup of
$M'(\R)$, so after replacing $g$ with $gm$ for some $m\in M'(\R)$ we
may assme $g\in \Norm_{G(\C)}(H'(\R))$.
It is well known that 
$$
\Norm_{G(\C)}(H'(\R))=\Norm_{M'(\C)}(H'(\R))\Norm_{G(\R)}(H'(\R)).
$$
For example see 
\cite{ic4}*{Proposition 3.12} (where the group in question is
denoted $W(R)^\theta$), or \cite{shelstad_innerforms}*{Theorem 2.1}.
Since conjugation by $M'(\C)$ does not change  infinitesimal or central
characters, by 
Proposition \ref{p:dslpacket} it preserves $\Pi_{M'}(\phi)$.
(See Lemma \ref{l:natural}). 
As above $G(\R)$ has no effect after the inductive step. 
This completes the proof.
\end{proof}

We now give the formula for the central character of $\Pi_G(\phi)$.
This follows immediately from the preceding discussion, and \eqref{e:kappa} applied to $M$.

\begin{lemma}
\label{l:charforchi}
Write $\phi$ as in \eqref{e:phi}(a) and (b), and suppose $h=\exp(2\pi
i\mu)$, with $\mu\in X_*(\ch H_0)\otimes\C\simeq X^*(H_0)\otimes\C$.
Let $\rho_i$ be one-half the sum of any set of 
positive roots of $\{\alpha\,|\,\ch\theta\alpha=-\alpha\}$, 
Set 
$$
\tau=\frac12(1-\ch\theta)\lambda-(1+\ch\theta)\mu+\rho_i\in X^*(H_0).
$$
\end{lemma}
Then the central character of $\Pi_G(\phi)$ is $\tau|_{Z(G(\R))}$. 

For the local Langlands classification to be well-defined it should be
natural with respect to automorphisms of $G$. This is the content of
the next lemma.

Suppose $\tau\in\Aut(G)$ is an involution which commutes with $\theta$. 
The automorphism $\tau$ acts on the pair $(\g,K)$, and 
defines an involution on the set of irreducible $(\g,K)$-modules, 
which preserves L-packets.

On the other hand, consider the image of 
$\tau$ under the sequence of maps
$\Aut(G)\rightarrow\Out(G)\simeq \Aut(D_b)\rightarrow\Aut(\ch G)$;
the final arrow is the transpose \eqref{e:transpose}. 
This extends to an automorphism of $\LG$ which we denote $\tau^t$. 
For example $\tau^t=1$ if and only if $\tau$ is an inner automorphism.

\begin{lemma}
\label{l:natural}
Suppose $\phi$ is an admissible homomorphism. Then
$\Pi_G(\phi)^\tau=\Pi_G(\tau^t\circ\phi)$.
\end{lemma}

\begin{remark}
Suppose $\tau$ is an inner automorphism of $G=G(\C)$.  It may not be
inner for $K$, and therefore it may act nontrivially on irreducible
representations. So it isn't entirely obvious that $\tau$ preserves
L-packets (which it must by the lemma). 

For example  $\iNT(\diag(i,-i))$ normalizes $SL(2,\R)$, 
and $K(\R)=SO(2)$, and
interchanges the two discrete series representations in an L-packet.
\end{remark}

\begin{proof}
This is straightforward from our characterization of the
correspondence.  Suppose $\tau$ is inner.
Then it  preserves infinitesimal  and radical
characters, and commutes with parabolic induction.
Therefore it preserves L-packets.
On the other hand  $\tau^t=1$.

In general, this shows (after modifying $\tau$ by an inner
automorphism) we may assume $\tau$ is distinguished. Then it is easy
to check the assertion on the level of infinitesimal and radical
characters, and it commutes with parabolic induction. The result
follows. We leave the details to the reader.
\end{proof}

\sec{Contragredient}
\label{s:contragredient}

Suppose $(\pi,V)$ is a $(\g,K)$-module.
Define a$(\g,K)$-module structure on $\Hom_\C(V,\C)$
by
\begin{subequations}
\label{e:contragredient}
\renewcommand{\theequation}{\theparentequation)(\alph{equation}}  
\begin{equation}
\pi^*(X)(f)(v)=-f(\pi(X)v)\quad (v\in V,X\in\g)
\end{equation}
and
\begin{equation}
\pi^h(k)(f)(v)=f(\pi(k\inv)v)\quad (v\in V,k\in K).
\end{equation}
\end{subequations}
Then the {\it dual} of $(\pi,V)$ is defined to be the $(\g,K)$-module
$(\pi^*,V^*)$,
where $V^*\subset \Hom_\C(V,\C)$ is the subspace of $K$-finite functionals.
See \cite{vogan_green}*{Definition 8.5.1}. 

Following \cite{chevalley_III} we say that a $(\g,K)$-module $(\pi',V')$ is contragredient to
$(\pi,V)$ if there is a nondegenerate bilinear form $B\colon V\times
V'\rightarrow\C$, respecting the actions of $\g$ and $K$.  If
$(\pi,V)$ is irreducible then $(\pi',V')$ is contragredient to
$(\pi,V)$ if and only if $(\pi',V')$ is isomorphic to the dual
$(\pi^*,V^*)$. Consequently we follow the common practice of using the
terms ``dual'' and ``contragredient'' interchangeably.

The proof of Theorem \ref{t:main} is now straightforward.
We restate the theorem here.

\begin{theorem}
\label{t:main2}
Let $G(\R)$ be the real points of a connected reductive algebraic
group defined over $\R$, with L-group $\LG$.
Let $C$ be the Chevalley involution of $\LG$ (Section \ref{s:aut})
and let $\tau$ be the Chevalley involution of $W_\R$ (Lemma \ref{l:tau}).
Suppose $\phi\colon W_\R\rightarrow \LG$ is an admissible homomorphism, 
with associated L-packet $\Pi(\phi)$. Let
$\Pi(\phi)^*=\{\pi^*\,|\,\pi\in\Pi(\phi)\}$. 
Then:

\medskip

\noindent (a) $\Pi(\phi)^*=\Pi(C\circ\phi)$

\smallskip

\noindent (b) $\Pi(\phi)^*=\Pi(\phi\circ\tau)$
\end{theorem}

\begin{proof}
Let $\caP=(\ch B_0,\ch H_0,\{X_{\ch\alpha}\})$ be the pinning used to 
define $\LG$.
After conjugating by $\ch G$ we may assume $C=C_\caP$ (Proposition \ref{p:CH}).
As in the discussion before \eqref{e:phi}, 
we are free to conjugate $\phi$ so that $\phi(\C^*)\in \ch H_0$, 
and $\phi(W_\R)\subset\Norm_{\ch G}(\ch H_0)$.

First assume $\Pi(\phi)$ is an L-packet of relative discrete series
representations. 
Then $\Pi(\phi)$ is determined by its infinitesimal character
$\chiinf(\phi)$ and its radical character $\chirad(\phi)$ (Definition
\ref{d:dslpacket}). It is easy to see that the infinitesimal character
of $\Pi(\phi)^*$ is $-\chiinf(\phi)$, and the 
radical character is $\chirad(\phi)^*$. 
So it is enough to show
$\chiinf(C\circ\phi)=-\chiinf(\phi)$ and
$\chirad(C\circ\phi)=\chirad(\phi)^*$.
The first is obvious from \eqref{e:phi}(a), the definition of
$\chiinf(\phi)$, and 
the fact that $C$ acts by $-1$ on the Lie
algebra of $\ch H_0$. 
The second follows from the fact that $C$ factors to the
Chevalley involution of $\L\Grad$, and the torus
case (Lemma \ref{l:Ctorus}).

Now suppose $\phi$ is any admissible homomorphism such that
$\Pi_G(\phi)$
is nonempty.
As in Definition \ref{d:lpacket} we may assume $\phi(W_\R)\subset\d M$ where $\d M$
is a standard Levi subgroup of $\LG$. Choose $M'$ as in Definition
\ref{d:general} and write $\Pi_{M'}=\Pi_{M'}(\phi)$ as in that Definition.

Write socle (resp. co-socle) for the set of irreducible submodules
(resp. quotients) of an admissible representation.

Choose $P=M'N$ as in Definition \ref{d:lpacket} to define
\begin{subequations}
\renewcommand{\theequation}{\theparentequation)(\alph{equation}}  
\label{e:ind}
\begin{equation}
\Pi_G(\phi)=\cosocle(\Ind_{P(\R)}^{G(\R)}(\Pi_{M'}))
=\bigcup_{\pi\in\Pi_{M'}}\cosocle(\Ind_{P(\R)}^{G(\R)}(\pi)).
\end{equation}

It is immediate from the
definitions that $C_\caP$ 
restricts to the Chevalley involution 
of $\ch M$. Therefore by the preceding case
$\Pi_{M'}(C\circ\phi)=\Pi_{M'}(\phi)^*$.
Compute $\Pi_G(C\circ\phi)$ using Definition \ref{d:lpacket};
this time the positivity condition in Definition \ref{d:lpacket} forces us to
to use the opposite parabolic $\overline
P=M'\overline N$:
\begin{equation}
\Pi_G(C\circ\phi)=\cosocle(\Ind_{\overline P(\R)}^{G(\R)}(\Pi_{M'}(\phi)^*)).
\end{equation}
Here is the proof of part (a), we justify the steps below.
\begin{equation}
\begin{aligned}
\Pi_G(C\circ\phi)^*&=[\cosocle(\Ind_{\overline P(\R)}^{G(\R)}(\Pi_{M'}(\phi)^*))]^*\\
&=
[\cosocle(\Ind_{\overline P(\R)}^{G(\R)}(\Pi_{M'}(\phi))^*)]^*\\
&=
[\socle(\Ind_{\overline P(\R)}^{G(\R)}(\Pi_{M'}(\phi)))^*]^*\\
&=
\cosocle(\Ind_{P(\R)}^{G(\R)}(\Pi_{M'}(\phi)))=\Pi_G(\phi).
\end{aligned}
\end{equation}
\end{subequations}

The first step is just the contragredient of \eqref{e:ind}(b). For the 
second, integration over $G(\R)/P(\R)$ is a pairing
between $\Ind_{P(\R)}^{G(\R)}(\pi^*)$ and $\Ind_{P(\R)}^{G(\R)}(\pi)^*$,
and gives
\begin{equation}
\label{e:indstar}
\Ind_{P(\R)}^{G(\R)}(\pi^*)\simeq
\Ind_{P(\R)}^{G(\R)}(\pi)^*.
\end{equation}
For the next step use $\cosocle(X^*)=\socle(X)^*$.
Then the double dual cancels for irreducible representations, and
it is well known that the theory of intertwining operators gives:
\begin{equation}
\socle(\Ind_{\overline P(\R)}^{G(\R)}(\pi))=
\cosocle(\Ind_{P(\R)}^{G(\R)}(\pi)).
\end{equation}
Finally plugging in \eqref{e:ind}(a) gives part (a) of the theorem.

For (b) we show that
$C\circ\phi$ is $\ch G$-conjugate to $\phi\circ\tau$. 

Recall $\tau$ is any automorphism of $W_\R$ acting by inverse on
$\C^*$, and any two such $\tau$ are conjugate by $\iNT(z)$ for $z\in\C^*$. 
Therefore the statement is independent of the choice of $\tau$.
It is convenient to choose $\tau(j)=j\inv$, i.e.,
$\tau=\tau_{-1}$ in the notation of the proof of Lemma \ref{l:tau}.

By \eqref{e:phi}(a) $(C\circ\phi)(z)=C(\phi(z))=z^{-\lambda}\overline
z^{-\lambda'}$. Recall (Lemma \ref{l:tau}) that $\tau(z)=z\inv$ for all $z\in\C^*\subset W_\R$, 
so $(\phi\circ\tau)(z)=\phi(z\inv)=
z^{-\lambda}\overline
z^{-\lambda'}$.
Therefore it is enough to show $C(\phi(j))$ is $\ch H_0$-conjugate to 
$\phi(\tau(j))$, which equals $\phi(j)\inv$ by our choice of $\tau$.

Since $\phi(j)$ normalizes $\ch H_0$, $\phi(j)=g\ch\delta$ 
with $g\in \Norm_{\ch G}(\ch H_0)$. 
Then $g\ch\delta(g)\phi(j)^2=\phi(-1)\in\ch H_0$. 
Therefore the  image $w$ of $g$ in $W$ satisfies $w\theta_0(w)=1$. 
Apply Lemma \ref{l:tits}(b).
\end{proof}

\begin{remark}
\label{r:dual}
The main result of \cite{real_chevalley}, together with Lemma 
\ref{l:natural}, gives an alternative proof of Theorem \ref{t:main}. 
By \cite{real_chevalley}*{Theorem 1.2}, there is a ``real'' Chevalley involution  $C_\R$ of $G$, which is defined over $\R$. This satisfies: $\pi^{C_\R}\simeq\pi^*$ 
for any irreducible representation $\pi$.
The transpose automorphism $C_\R^t$ of $\LG$ of Lemma \ref{l:natural}
is the Chevalley automorphism of $\LG$.
Then Lemma \ref{l:natural}  applied to  $C_\R$ implies Theorem \ref{t:main}.
\end{remark}
\sec{Hermitian Dual}
\label{s:hermitian}

Suppose $\pi$ is an admissible representation of $G(\R)$. 
We briefly recall what it means for $\pi$ to have an invariant
Hermitian form, and the notion of the 
Hermitian dual of
$\pi$. See \cite{knappvogan} for details, and for the connection with
unitary representations.

Let $\sigma$ be the antiholomorphic involution of $\G$ with fixed
points $\g(\R)=\Lie(G(\R))$.  We can and do assume that $\sigma$
commutes with the Cartan involution $\theta$, and write $\sigma$ for
the corresponding automorphism of $K(\C)$, so $K(\C)^\sigma$ is a
maximal compact subgroup of $K(\R)$.

We say a $(\g,K)$-module $(\pi,V)$, or simply $\pi$, is Hermitian if there is a nondegenerate Hermitian form
$(\,,\,)$ on $V$, satisfying
\begin{subequations}
\renewcommand{\theequation}{\theparentequation)(\alph{equation}}  
\begin{equation}
(\pi(X)v,w)+(v,\pi(\sigma(X))w)=0\quad(v,w\in V,X\in\g),
\end{equation}
and
\begin{equation}
(\pi(k)v,\pi(\sigma(k))w)=(v,w)\quad(v,w\in V, k\in K).
\end{equation}
\end{subequations}

Define the Hermitian dual $(\pi^h,V^h)$ as follows (compare 
\eqref{e:contragredient}). 
Define a representation of $\g$ on the space of conjugate-linear
functions $V\rightarrow\C$ by
\begin{subequations}
\renewcommand{\theequation}{\theparentequation)(\alph{equation}}  
\begin{equation}
\pi^h(X)(f)(v)=-f(\pi(\sigma(X))v)\quad (v\in V,X\in\g)
\end{equation}
and
\begin{equation}
\pi^h(k)(f)(v)=f(\pi(\sigma(k\inv))v)\quad (v\in V,k\in K).
\end{equation}
\end{subequations}
Let $V^h$ be the $K$-finite functions; then $(\pi^h,V^h)$ is a $(\g,K)$-module.
If 
$\pi$ is irreducible then $\pi$ is Hermitian if and only if 
$\pi\simeq\pi^h$.

Fix a Cartan subgroup $H$ of $G$.
Identify an infinitesimal character $\chiinf$ with (the Weyl group orbit of) 
an element
$\lambda\in\Hom(\h,\C)$, by the Harish-Chandra
homomorphism. 
Define $\lambda\rightarrow\overline\lambda$ with respect to the real
form $X^*(H)\otimes\R$ of $\Hom(\h,\C)$, and write $\overline{\chiinf}$ for the
corresponding action on infinitesimal characters. This is
well-defined, independent of all choices.

For simplicity we restrict to $GL(n,\R)$ from now on.

\begin{lemma}
Suppose $\pi$ is an admissible representation of $GL(n,\R)$, admitting
an infinitesimal character $\chiinf(\pi)$, and a central character
$\chi(\pi)$.
Then:
\begin{enumerate}
\item[(a)] $\chiinf(\pi^h)=-\overline{\chiinf(\pi)}$,
\item[(b)] $\chi(\pi^h)=\chi(\pi)^h$,
\item[(c)] Suppose $P(\R)=M(\R)N(\R)$ is a parabolic subgroup of $GL(n,\R)$, 
and $\pi_M$ is an  admissible representation of
$M(\R)$. Then
$\Ind_{P(\R)}^{G(\R)}(\pi_M^h)\simeq\Ind_{P(\R)}^{G(\R)}(\pi_M)^h$.
\end{enumerate}
\end{lemma}
In (b) $\chi(\pi)^h$ refers to the Hermitian dual of 
the one-dimensional representation of $Z(G(\R))=\R^*$. 

\begin{proof}
The first assertion is easy if $\pi$ is a minimal principal series
representation. Since any irreducible representation embeds in a
minimal principal series, (a) follows.
Statement (b) is elementary, and (c) is an easy variant of
\eqref{e:indstar}. We leave the details to the reader.  
\end{proof}

Now suppose $\phi$ is a finite-dimensional  representation of
$W_\R$ on a complex vector space $V$. Define a 
representation $\phi^h$ on the space $V^h$ of  conjugate linear functions
$V\rightarrow\C$ by:
$\phi(w)(f)(v)=f(\phi(w\inv) v)$
$(w\in W_\R,f\in V^h, v\in V)$. 
Choosing dual bases of $V,V^h$, 
identify $GL(V),GL(V^h)$ with $GL(n,\C)$, to write
$\phi^h=\phantom{}^t\overline{\phi}\inv$.

It is elementary  that $\phi$ has a nondegenerate invariant Hermitian form
if and only if
$\phi\simeq\phi^h$. 

Recall from the introduction that irreducible admissible
representations of $GL(n,\R)$ are  parametrized by equivalence classes
of $n$-dimensional semisimple representations of $W_\R$. Write
$\phi\rightarrow\pi(\phi)$ 
for this correspondence.

\begin{lemma}
Suppose $\phi$ is an $n$-dimensional semisimple representation of $W_\R$. 
Then
\begin{enumerate}
\item $\chiinf(\phi^h)=-\overline{\chiinf(\phi)}$,
\item $\chirad(\phi^h)=\chirad(\phi)^h$.
\end{enumerate}
\end{lemma}

\begin{proof}
Let $\ch H$ be the diagonal torus in $GL(n,\C)$.
As in \eqref{e:phi}, write $\phi(z)=z^\lambda\overline z^{\lambda'}$, 
so $\chiinf(\phi)=\lambda$.
On the other hand $\phi^h(z)=\overline{z^\lambda\overline
  z^{\lambda'}}\inv$,
and it is easy to see this equals
$z^{-\overline{\lambda'}}\overline z^{-\overline\lambda}$,
where $\overline\lambda$ is complex conjugatation with respect to
$X_*(\ch H)\otimes\R$. Therefore $\chiinf(\phi^h)=-\overline{\lambda'}$. 
Then (1) follows from the fact that, by \eqref{e:phi}(d), $\lambda$ is $GL(n,\C)$-conjugate
to $\lambda'$.

The second claim comes down to the case of tori, which we leave to the
reader.

\end{proof}

\begin{proof}[Proof of Theorem \ref{t:hermitian}]
The 
equivalence of (1) and (2) follow from the preceding discussion. 
For (3), it is well known (and a straightforward exercise)
 that $\pi(\phi)$ is tempered if and only if $\phi(W_\R)$ is
bounded \cite{borelCorvallis}*{10.3(4)},
which is equivalent to $\phi$
being unitary.

The proof of (1) is parallel to that of Theorem \ref{t:main}, 
using the previous two lemmas, and 
our characterization of the Langlands classification
in terms of infinitesimal character, radical character, and compatibility with
parabolic induction. We leave the few remaining details to the reader.
\end{proof}

\bibliographystyle{plain}
\def\cprime{$'$} \def\cftil#1{\ifmmode\setbox7\hbox{$\accent"5E#1$}\else
  \setbox7\hbox{\accent"5E#1}\penalty 10000\relax\fi\raise 1\ht7
  \hbox{\lower1.15ex\hbox to 1\wd7{\hss\accent"7E\hss}}\penalty 10000
  \hskip-1\wd7\penalty 10000\box7}
  \def\cftil#1{\ifmmode\setbox7\hbox{$\accent"5E#1$}\else
  \setbox7\hbox{\accent"5E#1}\penalty 10000\relax\fi\raise 1\ht7
  \hbox{\lower1.15ex\hbox to 1\wd7{\hss\accent"7E\hss}}\penalty 10000
  \hskip-1\wd7\penalty 10000\box7}
  \def\cftil#1{\ifmmode\setbox7\hbox{$\accent"5E#1$}\else
  \setbox7\hbox{\accent"5E#1}\penalty 10000\relax\fi\raise 1\ht7
  \hbox{\lower1.15ex\hbox to 1\wd7{\hss\accent"7E\hss}}\penalty 10000
  \hskip-1\wd7\penalty 10000\box7}
  \def\cftil#1{\ifmmode\setbox7\hbox{$\accent"5E#1$}\else
  \setbox7\hbox{\accent"5E#1}\penalty 10000\relax\fi\raise 1\ht7
  \hbox{\lower1.15ex\hbox to 1\wd7{\hss\accent"7E\hss}}\penalty 10000
  \hskip-1\wd7\penalty 10000\box7} \def\cprime{$'$} \def\cprime{$'$}
  \def\cprime{$'$} \def\cprime{$'$} \def\cprime{$'$} \def\cprime{$'$}
  \def\cprime{$'$} \def\cprime{$'$}

\end{document}